\documentclass[11pt,a4paper]{article}

\usepackage{amsfonts, amsmath, wasysym}
\usepackage{stmaryrd} 

\usepackage{tikz}
\usetikzlibrary{shapes, positioning}

\usepackage{amssymb,amsthm,
paralist
}

\usepackage{
latexsym,
}

\usepackage{url}

\definecolor{darkgreen}{rgb}{0,0.5,0}
\definecolor{darkred}{rgb}{0.7,0,0}
\usepackage[colorlinks, 
citecolor=darkgreen, linkcolor=darkred
]{hyperref}

\theoremstyle{plain}

\numberwithin{equation}{section}

\newcommand{\pl}[2]{{\frac{\partial #1}{\partial #2}}}

\newcommand{\ti}{\tilde}
\newcommand{\al}{\alpha}
\newcommand{\be}{\beta}

\newcommand{\de}{\delta}

\newcommand{\la}{\lambda}

\newcommand{\si}{\sigma}

\renewcommand{\th}{\theta}

\newcommand{\vph}{\varphi}
\newcommand{\ep}{\varepsilon}

\newcommand{\R}{\ensuremath{{\mathbb R}}}

\newcommand{\Z}{\ensuremath{{\mathbb Z}}}
\newcommand{\C}{\ensuremath{{\mathbb C}}}

\newcommand{\weakto}{\rightharpoonup}
\newcommand{\downto}{\downarrow}

\newcommand{\tensor}{\otimes}
\newcommand{\lap}{\Delta}

\newcommand{\grad}{\nabla}

\newcommand{\union}{\cup}

\DeclareMathOperator{\id}{id}

\newcommand{\beq}{\begin{equation}}
\newcommand{\eeq}{\end{equation}}
\newcommand{\beqa}{\begin{equation}\begin{aligned}}
\newcommand{\eeqa}{\end{aligned}\end{equation}}
\newcommand{\brmk}{\begin{rmk}}
\newcommand{\ermk}{\end{rmk}}
\newcommand{\partref}[1]{\hbox{(\csname @roman\endcsname{\ref{#1}})}}
\newcommand{\half}{\frac{1}{2}}

\newcommand{\lie}{{\mathcal L}}

\newcommand{\Rm}{{\mathrm{Rm}}}
\newcommand{\Ric}{{\mathrm{Ric}}}

\usepackage{soul}

 \newtheorem{thm}{Theorem}[section]

\newtheorem{lem}[thm]{Lemma}

\newtheorem{defn}[thm]{Definition}
\newtheorem{rmk}[thm]{Remark}

\newtheorem{claim}[thm]{Claim}

\newcommand{\norm}[1]{\left\Vert#1\right\Vert}
\newcommand{\abs}[1]{\left\vert#1\right\vert}

\title{\sc all two-dimensional \\ expanding ricci solitons}
\author{Luke T. Peachey and Peter M. Topping}
\date{5 December 2024}

\begin{document}

%

\parskip 8pt
\parindent 0pt

\maketitle

\begin{abstract}
The second author and H. Yin \cite{TY3} have developed a Ricci flow existence theory that gives
a complete Ricci flow starting with a  surface equipped with a conformal structure
and a nonatomic Radon measure as a 
volume measure.
This led to the discovery of a large array of new expanding Ricci solitons \cite{TY3}. 
In this paper we use the recent uniqueness theory in this context, also developed by the second author and H. Yin \cite{TY4}, to give a complete classification of all expanding Ricci solitons on surfaces.
Along the way, we prove a converse to the existence theory that is not constrained to solitons: every complete Ricci flow on a surface over a time interval $(0,\ep)$ admits a $t\downarrow 0$ limit within the class of admissible initial data.
This makes surfaces the first nontrivial setting for Ricci flow in which a bijection can be given between the entire set of complete Ricci flows over maximal time intervals $(0,T)$, and a class of initial data that induces them.
\end{abstract}

\section{Introduction}

Amongst all Ricci flows, the ones that evolve in a self-similar manner, modulo scaling and reparametrisation, are distinguished.
These so-called Ricci soliton flows can be classified as shrinking, steady or expanding. Shrinking solitons arise in the study of finite-time singularities of Ricci flow
(see e.g. \cite{EMT}, \cite[\S 11]{Perelman}, \cite{Bamler_solitons}), 
while expanding solitons are central to our understanding of both the large-time behaviour of the flow
(see e.g. \cite[Conjecture 16.8]{formations}, \cite{ChenZhu2000}, \cite{lott}) and the small-time asymptotics of the flow as it desingularises rough initial data (see e.g. \cite{GHMS}, \cite{SchulzeSimon2013}, \cite{deruelle_GAFA}, \cite{gianniotis_schulze}, \cite{lavoyer}).

Given a complete smooth Riemannian manifold
$(M,g)$,
always assumed to be connected,
and  a complete vector field $X$ on $M$,
let $\{\phi_t : M \rightarrow M\}_{t>0}$ be the family of diffeomorphisms generated by the time-dependent vector field $-\frac{X}{t}$, with $\phi_1 =\id_M$. Defining $g(t) := t \phi_t^*(g)$ for $t>0$, we can compute
\begin{equation*}
\begin{aligned}
\frac{\partial }{\partial t}g(t) &= \phi_t^*(g) - \phi_t^*\mathcal{L}_X(g)\\
&= \phi_t^* \big[g-\mathcal{L}_X(g)+2\Ric(g)\big] -2\Ric(g(t)),
\end{aligned}
\end{equation*}

We thus see that such a $g(t)$ is a Ricci flow if and only if $(M,g,X)$ is an expanding Ricci soliton in the following sense.

\begin{defn} 
\label{soliton_def1}
An expanding Ricci soliton is a triple $(M,g,X)$ where $(M,g)$ is a complete smooth Riemannian manifold, and $X$ is a complete smooth vector field on $M$, such that
\begin{equation}
\label{soliton_eq}
    2\Ric(g) - \mathcal{L}_X(g) + g = 0.
\end{equation}
We call the soliton trivial if $X$ is a Killing field.
\end{defn}
We will refer to the corresponding $g(t)= t \phi_t^*(g)$ as an expanding 
Ricci soliton flow.

The trivial expanding solitons are clearly just hyperbolic 
manifolds, suitably scaled.
Such solitons account for all expanding solitons on \emph{closed} 
manifolds
\cite{ivey}.

Special examples of nontrivial solitons can be found by imposing a symmetry ansatz and reducing the soliton equation \eqref{soliton_eq} to an ODE. For example, one can consider 
the $O(2)$-symmetric Ricci soliton flow starting with a two-dimensional cone
(see e.g. \cite[Section 4, Chapter 2]{ChowKnopf} or \cite{GHMS}).
In fact, this example is a so-called \emph{gradient} Ricci soliton, which means that $X$ is the gradient of some function, possibly modulo a Killing field. When considering solitons on surfaces, 
the condition of being a gradient soliton imposes a significant simplification because it turns out to \emph{force} the soliton to satisfy a symmetry ansatz and means that a complete classification can be obtained by solving ODEs (see e.g. \cite{Bernstein_Mettler} or 
\cite[Chapter 3]{chow_soliton_book}). In the case of solitons in two dimensions that are both gradient and expanding, one can therefore read off all possible examples. This classification will also 
arise as a by-product of
our analysis in Section \ref{grad_sect}. 

More sophisticated results about expanding solitons can be derived using PDE methods. 
We highlight the theorem of Deruelle \cite{deruelle_GAFA}, who showed that any Riemannian cone whose link is a smooth, simply connected, compact Riemannian manifold with curvature operator $\Rm\geq 1$ of general dimension can be evolved under Ricci flow as an expanding gradient Ricci soliton flow with non-negative curvature operator. 

%

Given earlier results such as Deruelle's theorem, one might develop intuition that expanding solitons tend to arise by running the Ricci flow starting with a cone. One aspect of our work, following earlier work of the second author and H. Yin \cite{TY3} is that this is very far from the full story.

The first objective  of this paper is to classify all expanding Ricci solitons, up to isometry, when the underlying manifold $M$ is two-dimensional.
For the purposes of this paper, we say that 
two 
expanding Ricci solitons $(M_1,g_1,X_1)$ and $(M_2,g_2,X_2)$ are isometric if there exists an isometry $\vph:(M_1,g_1)\to (M_2,g_2)$ with
$\vph_*X_1=X_2$.
One could instead replace this final condition on $X_1, X_2$ by the requirement that 
$\vph_*X_1-X_2$ 
is a Killing field, in which case further identifications of the solitons we find would be required.

The solitons we construct in two dimensions induce solitons in higher dimensions by taking products of finitely many of these examples and finitely many previously-known expanding solitons (including Euclidean space of arbitrary dimension). 

In order to classify  expanding solitons, we will construct a correspondence between nontrivial expanding Ricci  solitons and the following geometric structure that is easier to classify.
\begin{defn}
\label{meas_exp_def}
A \emph{measure expander} is a triple $(M,\mu,X)$
consisting of a smooth surface $M$ equipped with a conformal structure, 
a  complete conformal vector field $X$ on $M$, 
and a nontrivial nonatomic Radon measure $\mu$ on $M$
that is expanding under the action of $X$ in the sense that if we integrate the vector field  $X$ to give diffeomorphisms $\psi_s:M\to M$, $s\in\R$, with $\psi_0$ the identity, then 
\begin{equation}
\label{expanding_wrt_X}
    \lie_X \mu = \mu,\text{ where }\lie_X \mu := \frac{d}{ds} \psi_s^*(\mu)\bigg|_{s=0},
\end{equation}
or equivalently 
\begin{equation}
\label{expanding_wrt_X2}
\psi_s^*(\mu)=e^s\mu. 
\end{equation}
We say that two measure expanders 
$(M_1,\mu_1,X_1)$ and $(M_2,\mu_2,X_2)$ are \emph{isomorphic} if there exists a conformal diffeomorphism
$\vph:M_1\to M_2$ such that $\vph_*\mu_1=\mu_2$ and $\vph_* X_1=X_2$.
\end{defn}


Given a smooth surface $M$ with a conformal structure, 
and a  complete conformal vector field $X$, 
the correspondence will be between conformal Riemannian metrics
$g$ that make $(M,g,X)$ into a nontrivial expanding Ricci soliton and nontrivial nonatomic Radon measures $\mu$ on $M$ that make $(M,\mu,X)$ into a measure expander.
In order to explain this correspondence,
we need to survey the theory of Ricci flow on surfaces equipped with a conformal structure and a nonatomic Radon measure developed recently by the second author and Yin \cite{TY3, TY4}.
In the following theorem it is important to understand that the Ricci flow on surfaces preserves the conformal structure; if this conformal structure $[g(t)]$ is the same as the underlying conformal structure on $M$ we say 
that $g(t)$ is a \emph{conformal} Ricci flow.
We use the notation $\ti M$ for the universal cover of $M$, and $\ti\mu$ for the corresponding lift of $\mu$ to $\ti M$.

\begin{thm}[Well-posedness of Ricci flow from measures. From {\cite{TY3, TY4}}]
\label{exist_uniq_thm}
Let $M$ be a two-dimensional smooth manifold equipped with a conformal structure, 
and let $\mu$ be a Radon measure on $M$ that is nonatomic in the sense that
$$\mu(\{x\})=0\quad\text{ for all }x\in M.$$
Define $T\in [0,\infty]$ by
\begin{itemize}
\item
$T=\infty$ if $\ti M=D$, the unit disc in the plane;
\item
$T=\frac{1}{4\pi}\ti\mu(\ti M)$ if $\ti M=\C$;
\item
$T=\frac{1}{8\pi}\ti\mu(\ti M)$ if $\ti M=S^2$.
\end{itemize}
If $T\in (0,\infty]$,
then there exists a smooth complete conformal Ricci flow $g(t)$ on $M$, for $t\in (0,T)$,
attaining $\mu$ as initial data in the sense that 
$$\mu_{g(t)}\weakto \mu\text{ as }t\downto 0,$$
where $\mu_{g(t)}$ denotes the volume measure of the metric $g(t)$.
Moreover, 
whatever the value of $T\in [0,\infty]$,
if $\ti g(t)$, $t\in (0,\ti T)$, is any  smooth complete conformal Ricci flow on $M$ that attains $\mu$ as initial data in the same sense, then 
$\ti T\leq T$ and 
$$g(t)\equiv \ti g(t)\qquad\text{ for all }t\in (0,\ti T).$$
If $T\in (0,\infty)$ then $\mu_{g(t)}(M)=(1-\frac{t}{T})\mu(M)$ for all $t\in (0,T)$.
\end{thm}
The case $T=0$ corresponds to the situation that ${\ti M}$ is $\C$ or $S^2$ and $\mu$ is the trivial measure. In this case the theorem is indirectly asserting that there could never exist a complete conformal Ricci flow with this initial data.

This theory allows us to characterise expanding solitons in terms of measures using the following theorem, the first part of which 
is a refinement of the ideas in 
\cite{TY3}, based on \cite{TY4}.

\begin{thm}[Expanding solitons correspond to measure expanders]
\label{exp_meas_corresp}
%
Suppose $(M,\mu,X)$ is a measure expander in the sense of 
Definition \ref{meas_exp_def}.
Then the unique complete conformal Ricci flow $g(t)$ on $M$ given by Theorem \ref{exist_uniq_thm} with $\mu$ as initial data exists for all $t>0$.  Moreover, 
$(M,g(1),X)$ is a nontrivial  expanding Ricci soliton in the sense of 
Definition \ref{soliton_def1} and $g(t)$ is an expanding Ricci soliton flow.

Conversely, given any nontrivial two-dimensional expanding Ricci soliton 
$(M,g,X)$ (which automatically induces a conformal structure)
there exists a 
unique
nontrivial nonatomic Radon measure $\mu$ on $M$ that 
makes $(M,\mu,X)$ a measure expander
and induces $(M,g,X)$ in the sense above when we run the Ricci flow until time $t=1$.
\end{thm}

Note that although $X$ is not constrained explicitly to be nontrivial (for example) above, it will necessarily have to be so in order for $\mu$ to be expanding with respect to it.

Theorem \ref{exp_meas_corresp} reduces the nontrivial expanding soliton classification problem 
in two dimensions to the problem of finding all surfaces $M$ with a conformal structure, nontrivial complete conformal vector field $X$ and nontrivial nonatomic measure $\mu$ that is expanding with respect to $X$.
The task of finding such $(M,\mu,X)$ was initiated in \cite{TY3}, but we now give 
a full classification up to isomorphism as in Definition
\ref{meas_exp_def}.
%


%

Before we give the complete list of measure expanders, some examples are in order. Probably the simplest example is to work on the plane and consider the Lebesgue measure scaled by $e^x$, with $X=\pl{}{x}$. We can equally well take the quotient by a group of vertical translations generated by $(x,y)\mapsto (x,y+2\pi \al)$. This latter example is identifiable as the cone of cone angle $\pi \al$, without the vertex at $x=-\infty$. That is, if we view this cone conformally on the cylinder then its volume measure is the  measure
we are considering, and the soliton that flows out of it is the well-known soliton that has the topology of a cylinder, with 
one end asymptotic to a hyperbolic cusp (scaled to have curvature $-\half$) and with the other end 
asymptotic to a cone of angle ${\pi \al}$.
See, for example, \cite[Theorem 1.1, 3c]{Bernstein_Mettler}.
These examples are extremely special cases of families (Bi) and (Bii) in Theorem \ref{meas_exp_class} below.

Meanwhile, we could take one of these cones and add in the vertex to generate a different type of soliton. More precisely, given the volume measure of a cone, viewed on a cylinder as above, we can first change viewpoint by seeing the cylinder conformally as a punctured plane, and then add in the origin to change the underlying space. In this case the measure expander on the punctured plane is immediately a measure expander on the full plane, and the new soliton is the familiar soliton that smooths out the vertex of the cone
(see, for example, \cite[Theorem 1.1, 3a and 3b]{Bernstein_Mettler}
and the Gaussian soliton \cite[\S 1.2.2]{RFnotes}). This example is a very special case of family (Ci) in Theorem \ref{meas_exp_class} below.

In the following theorem, we classify measure expanders into families (A), (B) and (C).
Elements of family (A) will be conformally the disc.
Elements of families (B) and (C) will have the plane as their universal cover, and are distinguished by whether the vector field $X$ has a stationary point or not.
Each $M$ will either be of the form $\R\times N$, for 
$N$  one of $(0,\infty)$, $(0,\pi)$, $\R$ or $S^1=\R / 2\pi \Z$, where $N$ has the corresponding standard metric and $\R\times N$ has the corresponding standard conformal structure, or $M$ will be $\C$, 
viewed conformally as 
the cylinder $\R\times N$, for $N=S^1$, together with a point at $-\infty$ corresponding to the origin. As we proceed we will keep track of the isometry group of each $N$, which will be relevant afterwards to identify isomorphic measure expanders.
%
For any manifold $N$, we let $\mathcal{R}(N)$ denote the collection of nontrivial Radon measures on $N$.
%
We abuse notation by viewing volume forms  $dx$ and $e^{x}dx$ on $\R$ as Lebesgue measure $m$ and the weighted measure $m\llcorner e^x$ respectively.

\begin{thm}[Classification of measure expanders]
\label{meas_exp_class}
Every measure expander $(M,\mu,X)$ 
is isomorphic to an element of one of the following families of measure expanders:

\begin{enumerate}

\item[(Ai)] 
$$\left\{ \left(\mathbb{R} \times (0,\infty), e^x dx \otimes \nu, \frac{\partial}{\partial x} \right) \ \vert \ \nu \in \mathcal{R}((0,\infty)) \right\},$$
$N = (0,\infty)$ has trivial isometry group.

\item[(Aii)] 
$$\left\{ \left(\R \times (0,\pi), e^{\al x} dx \otimes \nu, \frac{1}{\al} \pl{}{x} \right) \ \vert \ \al>0,  \nu \in \mathcal{R}((0,\pi)) \right\},$$
$N = (0,\pi)$ has isometry group $\Z_2$.

\item[(Bi)]
$$\left\{ \left(\R \times \R, e^x dx \otimes \nu, \frac{\partial}{\partial x}\right) \ \vert \ \nu \in \mathcal{R}(\mathbb{R})\right\},$$
$N = \R$ has isometry group $\R \rtimes \Z_2$.

\item[(Bii)]
$$\left\{ \left(\R \times S^1, e^{\al x} dx \otimes \nu, \frac1{\al}\pl{}{x}
\right) \ \vert \ \al > 0, \nu \in \mathcal{R}(S^1) \right\},$$
$N = S^1$ has isometry group $O(2)$.

\item[(Biii)]
The same as (Bii), but with the conformal structure on $\R\times S^1$ being the pull-back of the standard product conformal structure by a twisting diffeomorphism $F_\be$ of $\R\times S^1$ defined by $F_\be(x,\th)=(x, \th+\be x)$, for some $\be>0$,
or equivalently
$$\left\{ \left(\R \times S^1,  (F_\be)_* (e^{\al x}dx \otimes \nu), \frac1{\al}\pl{}{x} + \frac{\be}{\al}\pl{}{\th}
\right) \ \vert \ \al > 0, \be > 0, \nu \in \mathcal{R}(S^1) \right\}.$$

$N = S^1$ has orientation-preserving isometry group $SO(2)$.

\item[(Ci)]
The same as (Bii) with the puncture at $x=-\infty$ filled in. In other words, the pushforward of any of the measure expanders from (Bii) under the injective conformal map 
\begin{equation}
\label{inj_conf_map}
\R \times S^1 \to \C, \quad (x , y) \mapsto e^{x+i y}.
\end{equation}
Alternatively, if we consider a complex coordinate
$z=e^{x+i\th}$ on $\C$, then we can write these as 
$$\left\{ \left(\C,
e^{\al x} dx \otimes \nu, \frac1{\al}\pl{}{x}
\right) \ \vert \ \al > 0, \nu \in \mathcal{R}(S^1) \right\}.$$

\item[(Cii)]
The same as (Biii) with the puncture at $x=-\infty$ filled in. In other words, the pushforward of any of the measure expanders from (Biii) under the injective conformal map 
\eqref{inj_conf_map}.
Alternatively, if we consider a complex coordinate
$z=e^{x+i\th}$ on $\C$, then we can write these as 
$$\left\{ \left(\C,
(F_{\be})_* (e^{\al x} dx \otimes \nu) , \frac1{\al}\pl{}{x} + \frac{\be}{\al} \pl{}{\th}
\right) \ \vert \ \al > 0, \be>0, \nu \in \mathcal{R}(S^1) \right\}.$$
\end{enumerate}

Moreover, any two measure expanders  $(M_1,\mu_1,X_1)$, $(M_2,\mu_2,X_2)$ in the above list are isomorphic if and only if they both belong to the same family in the list above (so $M_1 = M_2$), $X_1 = X_2$, and the measures on the fibres $\nu_1,\nu_2 \in \mathcal{R}(N)$ satisfy $$\varphi_*(\nu_1) = \lambda \cdot \nu_2,$$
for some $\lambda > 0$ and an isometry of the fibre $\varphi \in \textnormal{Isom}(N)$,
where $\vph$ is assumed to be orientation preserving in cases (Biii) and (Cii).
\end{thm}

\medskip

\begin{rmk}
To clarify, in parts (Ci) and (Cii), we push forward both the measure and the vector field to $\C$. In both cases the pushed forward vector field extends across the origin as a smooth complete vector field on the whole of $\C$ and we are left with new measure expanders on the larger space.
\end{rmk}

Implicit in Theorem \ref{meas_exp_class} is that every element of every family is indeed a measure expander. This will follow instantly from Corollary \ref{lem_split}, keeping in mind Remark \ref{not_nec_conf_rmk}.
It seems that no expanding soliton in the twisted family (Biii) has ever been considered before.

The second part of Theorem \ref{exp_meas_corresp} is proved by showing that 
after using a Ricci soliton $(M,g,X)$ to generate a Ricci soliton flow $g(t)$, we can take a limit of the flow as $t\downto 0$ to obtain a measure that encodes the soliton flow. In fact, we prove something much stronger in this paper that is nothing to do with solitons.

\begin{thm}[{Time zero  limits of complete Ricci flows}]
\label{time_zero_limit_thm}
Suppose $M$ is a smooth surface and $g(t)$ is any smooth complete  Ricci flow on $M$ for $t\in (0,\ep)$, for some $\ep>0$. 
Then there exists a nonatomic Radon measure $\mu$ on $M$ such that 
\begin{equation}
\label{mu_lim}
\mu_{g(t)}\weakto \mu
\end{equation}
as $t\downto 0$.
The measure $\mu$ is nontrivial unless the universal cover of $M$ is the disc and $g(t)=2th$ for $h$ a complete hyperbolic metric on $M$.
\end{thm}

We emphasise that for \emph{any} complete Ricci flow $g(t)$ on a surface, for $t\in (0,\ep)$, we are obtaining a $t\downto 0$ limit 
that is of the generality that the existence and uniqueness theorem 
\ref{exist_uniq_thm} can handle. 
For each smooth surface $M$, this is implying a one-to-one correspondence between
complete Ricci flows on $M$ over some time interval $(0,\ep)$ 
and the initial data that generates them.
More precisely, given $M$, we can consider initial data as
a combination of a conformal structure on $M$ and a nonatomic Radon measure on $M$ that is only permitted to be trivial in the case that the universal cover of $M$ with this conformal structure is the disc.
Our theory shows that the entire class of such initial data is in one-to-one correspondence with the space of 
complete Ricci flows on $M$ over maximal time intervals $(0,T)$, with the natural bijection arising by applying Theorem \ref{exist_uniq_thm}.
This begs the question as to what the analogue of our correspondence will be in higher dimensions.

\emph{Acknowledgements:} This work was supported by EPSRC grant EP/T019824/1.
The second author thanks Hao Yin for conversations on this topic.
For the purpose of open access, the authors have applied a Creative Commons Attribution (CC BY) licence to any author accepted manuscript version arising.

\section{Time zero limits of complete Ricci flows}

In this section we prove Theorem \ref{time_zero_limit_thm}. Many of the ingredients are refinements of estimates taken from \cite{TY3}. 
The argument we give for the existence of some limiting Radon measure (not necessarily nonatomic) is a streamlined version of an argument in \cite{peachey_thesis}.

We begin by noting that if we can prove the  existence of $\mu$ then the final claim, that $\mu$ is nontrivial except possibly if the universal cover of $M$ is the disc, follows immediately from Theorem \ref{exist_uniq_thm}.

We next  prove the existence of some Radon measure $\mu$ satisfying \eqref{mu_lim}, without worrying whether or not it is nonatomic.
It suffices to prove that each point $p\in M$ admits a
neighbourhood in which we can find such a limit. 
The global limit can then be constructed using a standard partition of unity argument.
Note that we are not fully localising the problem here because 
the remainder of the argument requires the global completeness of $g(t)$.

By lifting to the universal cover, we may as well assume that $M$ is either $\R^2$, $B_2$ or $S^2\simeq \C\union\{\infty\}$.
By adjusting by a biholomorphic transformation, we may as well assume that the point $p$ corresponds to the origin. 
We then task ourselves with proving \eqref{mu_lim} on the unit disc $D=B_1$.

The first step to achieve this will be to prove volume estimates. We need to verify that the volume $\mu_{g(t)}(D)$ is controlled from above for small time $t$:

\begin{claim}
\label{finite_vol_claim}
$$\limsup_{t\downto 0}\mu_{g(t)}(D)<\infty.$$
\end{claim}

This claim will follow if we can show that extremely large volume for very small time $t>0$ will imply very large volume for later times. 
This type of control was proved in \cite{TY3} when working on $\R^2$; we adjust the argument to work on $S^2\simeq \C\union\{\infty\}$.
Recall that $B_r$ is the Euclidean ball of radius $r$ centred at the origin.

\begin{lem}[cf. {\cite[Lemma 3.1]{TY3}}]
\label{lem_lower_vol_sphere}
Let $\mu$ be a nonatomic Radon measure on $S^2\simeq \C\union\{\infty\}$ and $g(t)$ for $t \in (0,T)$ be the complete Ricci flow on the sphere starting from $\mu$ as given by Theorem~\ref{exist_uniq_thm}.  Then for $0<r<R<\infty$ and 
$0< t < \frac{\mu(B_r)}{8\pi}$, 
we have
\begin{equation*}
\mu(B_r) \leq  \mu_{g(t)}(B_R) + \frac{8\pi t}{1 - (\frac{r}{R})^2}.
\end{equation*}
\end{lem}

We will repeatedly need the following theorem of the second author and Yin \cite{TY4} that establishes a maximally stretched property that substitutes for a maximum principle.

\begin{thm}[{\cite[Theorem 1.2]{TY4}}] 
\label{ordered_thm}
Let $M$, $\mu$ and $T$ be as in Theorem \ref{exist_uniq_thm}, and let 
$g(t)$ be the unique Ricci flow constructed in that theorem.
Suppose now that $\nu$ is any other Radon measure on $M$ with $\nu\leq \mu$,
and $\ti g(t)$, $t\in (0,\ti T)$, is any smooth conformal Ricci flow  on $M$
attaining $\nu$ as initial data. Then $\ti T\leq T$ and 
$$g(t)\geq \ti g(t)$$
for all $t\in (0,\ti T)$.
\end{thm}

\begin{proof}[Proof of Lemma \ref{lem_lower_vol_sphere}]
We may as well assume that $\mu(B_r)>0$ since otherwise the lemma is vacuous.
Since the restriction 
$\mu\llcorner B_r$
of $\mu$ to $B_r$ is also a nontrivial nonatomic Radon measure on the sphere, we can start the Ricci flow $G(t)$ on $S^2$ from this measure, for $t \in (0,\frac{\mu(B_r)}{8\pi})$, using Theorem \ref{exist_uniq_thm}. During this time we have $$\mu_{G(t)}(S^2) = \mu(B_r)-8\pi t.$$ By Theorem~\ref{ordered_thm}, 
applied  on $S^2$,
we have that $G(t) \leq g(t)$.
Writing $H_{r}$ for the unique complete conformal hyperbolic metric on $S^2 \setminus B_{r}$, we can also apply Theorem~\ref{ordered_thm} on $S^2 \setminus B_{r}$
to give $G(t) \leq 2t H_{r}$, since both flows start at the trivial measure and $2tH_r$ is complete.
%
Together these give the inequality
\begin{align*}
    \mu_{g(t)}(B_R) & \geq \mu_{G(t)}(B_R) = \mu_{G(t)}(S^2) - \mu_{G(t)}(S^2 \setminus B_R)\\ &\geq \mu(B_r) - 8\pi t - 2 t \mu_{H_r}(S^2 \setminus B_R).
\end{align*}
The result follows from the direct calculation
\begin{equation*}
    \mu_{H_r}(S^2 \setminus B_R) = \frac{4\pi r^2}{R^2-r^2}. \qedhere
\end{equation*}
\end{proof}

\begin{proof}[Proof of Claim \ref{finite_vol_claim}]
If not, then there exists a decreasing and null sequence $t_n$ such that
\begin{equation*}
    \lim_{n \rightarrow \infty} \mu_{g(t_n)}(D) = \infty.
\end{equation*}
Consider $n$ large enough so that $\mu_{g(t_n)}(D) > 8 \pi$.
By appealing to Theorem \ref{exist_uniq_thm}, we can extend $g(t)$, originally defined for $t\in (0,\ep)$, to at least the time interval $(0,t_n+1)$, for each $n$.
Let $\mu_n$ denote the restriction of the measure $\mu_{g(t_n)}$ to $D$, viewed as a measure on $S^2\simeq\C\union\{\infty\}$.
Let $G_n(t)$ be the Ricci flow on $S^2$ for $t \in (0,1]$ starting from $\mu_n$.  By Theorem~\ref{ordered_thm},
applied on $M$,
we have that $G_n(t) \leq g(t_n+t)$ for $t \in (0,1]$, and so by Lemma~\ref{lem_lower_vol_sphere} applied to $G_n(t)$, we have
\begin{align*}
\mu_{g(t_n)}(D) = \mu_n(D) &\leq  \mu_{G_n(1)}(B_{\sqrt{2}}) + 16 \pi\\
& \leq \mu_{g(1+t_n)}(B_{\sqrt{2}}) + 16 \pi.
\end{align*}
This yields a contradiction, as taking $n \rightarrow \infty$ makes the left-hand side diverge to infinity, whereas the 
right-hand side converges to $\mu_{g(1)}(B_{\sqrt{2}}) + 16 \pi < \infty$.
\emph{This completes the proof of Claim \ref{finite_vol_claim}.}
\end{proof}

By Claim \ref{finite_vol_claim}, 
combined with weak compactness for measures
\cite[Theorem 1.41]{evans_gariepy},
every sequence $t_n\downto 0$ has a subsequence so that 
$$\mu_{g(t_n)}\weakto \mu$$
for some Radon measure $\mu$ on $D$, as $n\to\infty$. We would like to establish that the sequence $t_n$ is not significant in that  $\mu_{g(t)}\weakto \mu$ as $t\downto 0$,
i.e. 
\beq
\label{psi_lim}
\int_D\psi d\mu_{g(t)}\to \int_D\psi d\mu\quad\text{ as }t\downto 0
\eeq
for all $\psi\in C_c^0(D)$.

\begin{claim}
\label{cts_limit_claim_new}
The limit \eqref{psi_lim} holds for all $\psi\in C_c^\infty(D)$.
\end{claim}

If Claim \ref{cts_limit_claim_new} holds then for all $\psi\in C_c^0(D)$ we can consider the mollification $\psi_h$ and estimate
\beqa
\abs{\int \psi d\mu_{g(t)} - \int \psi d\mu } &\leq 
\abs{\int \psi_h d\mu_{g(t)} - \int \psi_h d\mu } +\|\psi-\psi_h\|_{C^0}\big(\mu_{g(t)}(D)+\mu(D)\big).\\
\eeqa
Invoking Claims \ref{cts_limit_claim_new} and \ref{finite_vol_claim} gives
$$\limsup_{t\downto 0}\abs{\int \psi d\mu_{g(t)} - \int \psi d\mu }
\leq 0+C\|\psi-\psi_h\|_{C^0},$$
for $C$ independent of $h$ and $\psi$. 
Sending $h\downto 0$ then gives \eqref{psi_lim} for all $\psi\in C_c^0(D)$ as desired.

\begin{proof}[Proof of Claim~\ref{cts_limit_claim_new}]
We already know that 
$$\int_D\psi d\mu_{g(t_n)}\to \int_D\psi d\mu\quad\text{ as }n\to\infty$$
for our null sequence $t_n$, so it suffices to establish the integrability of the function
$$t\mapsto \frac{d}{dt}\int \psi d\mu_{g(t)}$$
over some small time interval $(0,\de)$,
because in that case we could estimate
$$\sup_{t\in (0,t_n)}\bigg|\int_D\psi d\mu_{g(t)}
-\int_D\psi d\mu\bigg|
\leq
\bigg|\int_D\psi d\mu_{g(t_n)}-\int_D\psi d\mu\bigg|
+\int_0^{t_n} \bigg|\frac{d}{ds}\int \psi d\mu_{g(s)}
\bigg|ds$$
and then send $n\to\infty$.
Writing $u(t)$ for the conformal factor of $g(t)$, with respect to the standard flat metric on $D$, we know that 
the function  $t\mapsto \frac{u(\cdot,t)}{t}$ is 
monotonically decreasing
(see, e.g. \cite[Remark 2.5]{TY4})
and so there exists some $\eta > 0$ such that $u(t) \geq \eta t$ on $D \times (0,\frac{\ep}{2}]$. 
The significance of $\frac{\ep}{2}$ is that it lies within the existence time interval $(0,\ep)$.
Using the same argument as in \cite[Lemma 4.5]{TY3}, we compute for our $\psi\in C_c^\infty(D)$ that
\begin{align*}
 \abs{ \frac{d}{dt} \int \psi d\mu_{g(t)}} = 
 \abs{ \frac{d}{dt} \int \psi u dx} =
 \abs{ \int \psi \lap\log u \,dx} &= \abs{ \int \Delta \psi \log u\, dx}\\
 & \leq \norm{\Delta \psi}_{C^0} \int_{D} |\log u(t)| dx. 
\end{align*}
We can then estimate for $\eta t\in (0,1]$ that
\begin{align*}
\int_{D} |\log u(t)| dx &=
\int_{D\cap \{u\geq 1\}} \log u(t) dx + \int_{D\cap \{u< 1\}} (-\log u(t)) dx\\
& \leq \int_{D} u(t)\, dx + \pi (-\log \eta t)\\
& \leq C + \pi (-\log \eta t)
\end{align*}
by Claim~\ref{finite_vol_claim},
which is integrable.
\emph{This completes the proof of Claim \ref{cts_limit_claim_new}.}
\end{proof}

We have shown that there exists a Radon measure $\mu$ on $D$ such that $\mu_{g(t)}\weakto \mu$ as $t\downto 0$.
It remains to show that $\mu$  is nonatomic, which will also require Lemma~\ref{lem_lower_vol_sphere}.
%
Suppose instead that $\mu$ has a point of positive mass. After modifying by an automorphism of the disc $D$, we may assume that
$\mu(\{0\}) = \delta > 0$. 
Choose $t_0 := \min\{ \frac{\delta}{64 \pi} , \frac{\ep}{2} \}$. We shall show that the volume measure at this positive time $t_0$ also has a point of positive mass, which will be a contradiction. 
For any $r\in (0,\half)$ we have
\begin{equation*}
    \liminf_{s \downto 0} \mu_{g(s)}(B_r) \geq 
    \mu(B_r)
    \geq\delta,
\end{equation*}
and hence for $s_0\in (0,t_0)$ sufficiently small  $$\mu_{g(s_0)}(B_r) \geq \frac{\delta}{2}.$$ If $\mu_r$ denotes the restriction of $\mu_{g(s_0)}$ to $B_r$, 
seen as a measure on $S^2\simeq\C\union\{\infty\}$,
then by Theorem~\ref{exist_uniq_thm}, there exists a Ricci flow $G_r(t)$ for $t \in (s_0,s_0+\frac{\delta}{16\pi})$ on $S^2$ starting weakly from $\mu_r$ at time $s_0$. Moreover, by Theorem~\ref{ordered_thm}, we have that $G_r(t_0) \leq g(t_0)$
on $M$.
Applying Lemma~\ref{lem_lower_vol_sphere} with $R = \sqrt{2} r$ gives
the second inequality in
\begin{equation*}
 \mu_{g(t_0)}(B_R) \geq \mu_{G_r(t_0)}(B_R) \geq \mu_{g(s_0)}(B_r) - 16 \pi t_0 \geq \frac{\delta}{4}, 
\end{equation*}
and therefore
\begin{equation*}
    \mu_{g(t_0)}(\{0\}) = \lim_{R \downto 0} \mu_{g(t_0)}(B_R) \geq \frac{\delta}{4}, 
\end{equation*}
which gives a contradiction and completes the proof that $\mu$ is nonatomic.
\emph{This completes the proof of Theorem~\ref{time_zero_limit_thm}.}

\section{Expanders versus measures}

In this section we prove Theorem~\ref{exp_meas_corresp}.

Suppose 
$(M,\mu,X)$ is a measure expander as in Definition \ref{meas_exp_def}.
%
Let $g(\tau)$ for $\tau \in (0,T)$ be the unique complete conformal Ricci flow on $M$ starting weakly from $\mu$, given by Theorem~\ref{exist_uniq_thm}.

Fixing a scaling parameter $t > 0$, we first note that the rescaled Ricci flow $\tau\mapsto t^{-1} g(t \tau)$ for $\tau \in (0,t^{-1} T)$ starts weakly from the rescaled measure $t^{-1} \mu$. 

Alternatively, let $\phi_t:M \rightarrow M$ be again the flow of the vector field $-\frac{X}{t}$ for all $t>0$, with $\phi_1=id_M$,
and let $\psi_s:M \rightarrow M$ be again the flow of the vector field $X$ for all $s\in\R$, with $\psi_0=id_M$, so
$$\psi_s=\phi_{e^{-s}}.$$
The fact that $\mu$ is expanding with respect to $X$, i.e. \eqref{expanding_wrt_X},
can be equivalently written as 
$\phi_t^*(\mu)
=\psi_{-\log t}^*(\mu)
=t^{-1}\mu$ for all $t>0$
as in \eqref{expanding_wrt_X2}.
Thus, the Ricci flow $\tau\mapsto\phi_{t}^*(g(\tau))$ for $\tau \in (0,T)$ also starts weakly from $\phi_{t}^*(\mu) = t^{-1} \mu$, and therefore, by Theorem~\ref{exist_uniq_thm}, the two flows starting with $t^{-1} \mu$
must agree:
$$t^{-1} g(t \tau)=\phi_{t}^*(g(\tau)).$$
In particular, by setting $\tau=1$, we can deduce that not only is our flow immortal, but it must also satisfy the relation
\begin{equation*}
    g(t) = t \phi_t^*(g(1)), \quad \forall t > 0,
\end{equation*}
so $(M,g(1),X)$ is an expanding Ricci soliton and $g(t)$ is an expanding Ricci soliton flow. Moreover, the soliton is nontrivial because otherwise $\mu$ would have to be trivial.

Conversely, suppose we start with a nontrivial expanding Ricci soliton $(M,g,X)$ with $g(t) = t \phi_t^*(g)$ the associated Ricci soliton flow. By Theorem \ref{time_zero_limit_thm}, there exists a nonatomic 
Radon measure $\mu$ such that $g(t)$ starts weakly from $\mu$. As above, we note that for any $s \in \mathbb{R}$
\begin{equation*}
   \psi_s^*(g(t)) 
   = \psi_s^*(t\psi_{-\log t}^*(g))
   =t\psi_{s-\log t}^*(g)
   = t \phi_{t e^{-s}}^*(g) = e^s g(t e^{-s}), \quad \forall t > 0,
\end{equation*}
and hence taking $t \downto 0$, we have that $\psi_s^*(\mu) = e^s \mu$, for any $s \in \mathbb{R}$, that is,  $\lie_X \mu = \mu$, which means that $\mu$ is expanding with respect to $X$.
We see that $\mu$ is nontrivial because, for example, Theorem \ref{time_zero_limit_thm} tells us that this is only possible when $(M,g,X)$ is trivial.

\section{Classification of measure expanders}
\label{class_meas_exp_sect}

In this section we prove the classification of measure expanders claimed in Theorem \ref{meas_exp_class}. In light of Theorem \ref{exp_meas_corresp}, 
and the discussion in the introduction of trivial solitons, this will
complete the classification of expanding Ricci solitons in two dimensions.

\subsection{Simply connected measure expanders}
\label{simp_conn_sect}

We begin with the case that $M$ is a simply connected smooth surface equipped with a conformal structure, which can only be 
$\C$, $D$ or $S^2$. The final case $S^2$ can instantly be discounted: 
$S^2$ cannot admit a measure expander because it would have to have finite total measure and so we could compute
$$\mu(S^2)=[\psi_s^*\mu](S^2)=e^s\mu(S^2),$$
for all $s\in \R$ by \eqref{expanding_wrt_X2}, which is a contradiction because $\mu$ is nontrivial.
Alternatively, any measure expander on $S^2$ would induce 
an expanding soliton, but all expanding Ricci solitons on closed surfaces have constant curvature $-\half$ as mentioned in the introduction.
This leaves us with the cases that $M=\C$ or $M$ is the disc,  equivalently a half-space.

The next ingredient for a simply connected measure expander is a nontrivial complete conformal vector field $X$. Such vector fields are highly constrained.
They integrate to give conformal automorphisms $\psi_s$.
On $\C$, such automorphisms are all of the form 
$z\mapsto  a z+ b$. 
Thus either $X$ has no zeros and we can pick a global complex coordinate $z=x+iy$ on $\C$ so that $X=\pl{}{x}$,
or $X$ has one zero and by choosing coordinates that make this the origin we have 
%
$X=\rho r\pl{}{r} + \si \pl{}{\th}$, for real $\rho,\si \in\R$, i.e. $X$ is an infinitesimal dilation/rotation.
%
%
%
%
%
To be part of a measure expander, 
we must have $\rho>0$, and 
we can then define $\al=\frac{1}{\rho}>0$.
This is because otherwise we would have
$\psi_s(D)\subset D$ for all $s\geq 0$, 
implying that $\psi_s^*\mu(D)\leq \mu(D)$ for $s\geq 0$, whereas for $\mu$ to be expanding with respect to $X$ we need that $\psi_s^*\mu(D)=e^s\mu(D)$.
Moreover, after modification by a reflection, we may assume that $\si \geq 0$.

Similarly, on the upper half-plane $\mathbb{H}$, after modification by an automorphism we can reduce to the cases that $X=\pl{}{x}$ 
or 
$X=\frac{r}{\al}\pl{}{r}$ for $\al>0$.


The final ingredient for a simply connected measure expander is a nontrivial nonatomic Radon measure expanding under the action of $X$.

\begin{lem}\label{lem_split}
Let $N$ be either $S^1$ or a connected open subset of $\R$, and let $\mu$ be a nontrivial nonatomic Radon measure on $\mathbb{R} \times N$. Fix $\al > 0$. Then $\mu$ is expanding under the action of the translating vector field $X = \frac{1}{\al} \frac{\partial}{\partial x}$ in the sense of \eqref{expanding_wrt_X} if and only if $\mu$ is a product measure
\begin{equation}\label{eqn prod}
    \mu = e^{\al x} dx \otimes \nu,
\end{equation}
where $dx$ denotes the Lebesgue measure on $\mathbb{R}$, and $\nu$ is a nontrivial (not necessarily nonatomic) Radon measure on $N$. 
\end{lem}

In practice, $N$ will be either $S^1$, $\R$, $(0,\infty)$ or $(0,\pi)$.

\begin{rmk}
\label{not_nec_conf_rmk}
It will be useful to keep in mind that Lemma~\ref{lem_split} does not mind with which conformal structure that $\R\times N$ is endowed, although in practice we will want the vector field to be conformal.
\end{rmk}

\begin{proof}[Proof of Lemma~\ref{lem_split}]
First note that the vector field $X$ integrates to give diffeomorphisms $\psi_s(x,y)=(x+\frac{s}{\al},y)$.
Given any nonatomic nontrivial Radon measure $\mu$ on  $\R \times N$, define a new nonatomic nontrivial Radon measure $\ti \mu$ via the relation $\mu = e^{\al x} \ti \mu.$ Since $\psi_s^*(e^{\al x} \ti \mu) = e^{\al x + s}  \psi_s^*(\ti \mu)$, we have that
$$ \lie_X \mu = \lie_X (e^{\al x} \ti \mu) = \frac{d}{ds} \psi_s^* (e^{\al x} \ti \mu)\vert_{s=0} =  \mu + e^{\al x}  (\lie_X \ti \mu).$$
In particular, $\mu$ is expanding (i.e, $\lie_X \mu = \mu$) if and only if $\ti \mu$ is invariant under $X$ (i.e, $\lie_X \ti \mu = 0$). Furthermore, since a measure $\ti \mu$ is invariant under horizontal translations if and only if it decomposes as a product measure
$$ \ti \mu = dx \otimes \nu,$$
where 
$\nu$ is the Borel measure on $N$ defined by
$$\nu(A) = \ti \mu(A \times [0,1)),$$
the result follows.
\end{proof}

\begin{lem}\label{lem sc meas}
Any simply connected measure expander is isomorphic to one of the measure expanders from cases (A), (Bi) or (C) in Theorem~\ref{meas_exp_class}.
\end{lem}

\begin{proof}[Proof of Lemma~\ref{lem sc meas}]
If $M = \mathbb{H}$, as discussed earlier we may assume that $X$ is $\pl{}{x}$ or $\frac{r}{\al} \pl{}{r}$. In the first case, applying Lemma~\ref{lem_split} with $N =(0,\infty)$, we must have a measure expander from (Ai). In the second case, we push forward by the conformal diffeomorphism $$\mathbb{H} \rightarrow \mathbb{R} \times (0,\pi), \quad z = r e^{i\theta} \mapsto (\log r , \theta),$$
to give a measure expander $(\R \times (0,\pi),\mu, \frac{1}{\al} \pl{}{x})$. Applying Lemma~\ref{lem_split} with $N = (0,\pi)$, we must have a measure expander from (Aii). These are all possible measure expanders on $\mathbb{H}$ up to isomorphism.

If $M = \C$ and $X = \pl{}{x}$, applying Lemma~\ref{lem_split} with $N = \R$, we have a measure expander from (Bi).

Finally, we need to consider the case $M = \C$ and $X = \frac{r}{\al} \pl{}{r} + \frac{\be}{\al} \pl{}{\th}$, for some $\al > 0$ and $\be \geq 0$. Note that if we remove the origin, then we still have a measure expander, and this new measure expander uniquely determines the original. Pulling back under the conformal transformation 
$$ \R \times S^1 \to \C \setminus \{0\}, \quad (x,\th) \mapsto e^{x+i\th},$$
we have a measure expander $(\R \times S^1 , \mu , \frac{1}{\al} \pl{}{x} + \frac{\be}{\al} \pl{}{\th})$. If $\be = 0$, we can immediately apply Lemma~\ref{lem_split} with $N=S^1$ to conclude that our punctured measure expander is in case (Bii), and hence our original measure expander is in case (Ci). Otherwise $\be > 0$, so that when we pull back by the twisting diffeomorphism 
$F_{\be}(x,\th)=(x,\th+\be x)$ from
Theorem~\ref{meas_exp_class}, we have a measure expander $$(F_{\be}^*(\R \times S^1) , F_{\be}^*\mu, \frac{1}{\al} \pl{}{x}),$$
to which we can apply Lemma~\ref{lem_split}, giving $F_{\be}^*\mu = e^{\al x} dx \otimes \nu$, for some $\nu \in \mathcal{R}(S^1)$. Here we use the notation $F_{\be}^*(\R \times S^1)$ to indicate that we are pulling back the conformal structure of $\R \times S^1$.

Pushing forward by $F_{\be}$, our punctured measure expander 
$$\left(\R \times S^1 , (F_{\be})_*(e^{\al x} dx \otimes \nu), \frac{1}{\al} \pl{}{x} + \frac{\be}{\al} \pl{}{\th}\right)$$ 
is in (Biii), and so our original measure expander is in (Cii). 
\end{proof}

At this point we have found all possible simply connected measure expanders.

\subsection{Quotients of measure expanders}

For any measure expander $(M,\mu,X)$, its universal cover $(\ti M, \ti \mu, \ti X)$ is also a measure expander. 
Moreover, the deck transformations of the covering $G$ form a discrete subgroup of the group of conformal automorphisms of $
\ti M$, such that all nontrivial elements of $G$ are fixed point free, and each element of $G$ preserves both $\ti \mu$ and $\ti X$:
$$ g_* \ti \mu = \ti \mu, \quad g_* \ti X =\ti X, \quad \forall g \in G.$$

In this section we identify all possible ways in which we can go the other way and take a quotient of one of the simply connected measure expanders found in Section \ref{simp_conn_sect} to give a new measure expander.

The only nontrivial conformal automorphisms of $\mathbb{H}$ preserving $\pl{}{x}$ are horizontal translations, but these would scale any measure from (Ai). Moreover, the only nontrivial conformal automorphisms of $\mathbb{H}$ preserving $\frac{r}{\al} \pl{}{r}$ are dilations, which after pushing forward
to $\R \times (0,\pi)$ via the biholomorphism 
$z=re^{i\th}\mapsto (\log r, \th)$, correspond to horizontal translations preserving the vector field $\frac{1}{\al} \pl{}{x}$. But any horizontal translation would scale a measure from (Aii). We also note that the only fixed point free automorphisms of $\C$ are translations. Therefore, up to isomorphism, we may assume that any measure expander that is not simply connected has universal cover $\ti M =\C$, with $\ti X = \pl{}{x}$, putting the universal cover in case (Bi). 

In order for a translation of $\C$ to preserve a measure $\ti \mu$ from case (Bi), it must have a nontrivial vertical component. We first consider the situation that such a translation is purely vertical, so that after dilating, we may assume our measure expander is of the form 
$$(\C , \ti \mu, \frac{1}{\al} \pl{}{x}),$$
for some $\al > 0$, and is invariant under the normalised translations
$$(x,y) \mapsto (x, y + 2\pi n), \quad n \in \Z.$$
Passing to the quotient we have a measure expander 
$$(\R \times S^1, \mu,  \frac{1}{\al} \pl{}{x}).$$
By either looking at the structure of $\ti \mu$ or applying Lemma~\ref{lem_split} directly, we see that $\mu$ is of the form
\begin{equation*}
\mu = e^{\al x} dx \otimes \nu,
\end{equation*}
for some $\nu \in \mathcal{R}(S^1)$, and therefore the quotiented measure expander is in case (Bii).

In the more general case that our translations are not orthogonal to our vector field, we first dilate and rotate to put our measure expander in the form 
$$(\C , \ti \mu, \frac{1}{\al} \pl{}{x} + \frac{\be}{\al} \pl{}{y}),$$
for some $\al > 0$ and $\be \in \R \setminus \{0\}$, and so it is invariant under the normalised translations 
$$(x,y) \mapsto (x , y + 2\pi n), \quad n \in \Z.$$
After possibly reflecting, we may also assume that $\be > 0$.

The quotient is then a measure expander $(\R \times S^1 , \mu, \frac{1}{\al} \pl{}{x} + \frac{\be}{\al} \pl{}{\th})$. Pulling back by the twisting diffeomorphism $F_{\be}$ as defined in Theorem~\ref{meas_exp_class} gives a measure expander $$((F_{\be})^*(\R \times S^1) , F_{\be}^*\mu, \frac{1}{\al} \pl{}{x}).$$
Applying Lemma~\ref{lem_split}, we find that 
\begin{equation*}
F_{\be}^*\mu = e^{\al x} dx \otimes \nu,
\end{equation*}
for some $\nu \in \mathcal{R}(S^1)$, and therefore the quotiented measure expander is in (Biii). This completes the proof of the first part of Theorem~\ref{meas_exp_class}.

\subsection{Isomorphic measure expanders}
We now know that every measure expander is isomorphic to a measure expander appearing in the list from Theorem~\ref{meas_exp_class}. It remains to establish when  two  measure expanders $(M_1,\mu_1,X_1)$, $(M_2,\mu_2,X_2)$ from the list are isomorphic.

We would like to show first that they belong to the same family (in particular we have $M_1=M_2$, not just modulo conformal automorphism) and the vector fields are identical.

It is clear that $M_1$ and $M_2$ have to have the same underlying conformal type, which immediately separates families (Ai) and (Aii) from the others.

To distinguish between (Ai) and (Aii), which are both conformal to the disc, observe that in case (Ai), the automorphisms generated by the vector field are parabolic,
i.e., they extend to bijections of $\overline{D}$ that fix
precisely one point on the boundary $\partial D$,
whereas in (Aii) they are hyperbolic, i.e., they extend to bijections of $\overline{D}$ that fix
precisely two points on the boundary $\partial D$.
If the two measure expanders are in case (Ai) then they automatically have the same vector field 
$X=\pl{}{x}$ by definition of that family.
If the two measure expanders are in case (Aii), with 
$X_1=\frac{1}{\al_1}\pl{}{x}$ and
$X_2=\frac{1}{\al_2}\pl{}{x}$, for $\al_1,\al_2>0$, then 
the only conformal automorphisms of the strip $\R \times (0,\pi)$ that could push forward $X_1$ to $X_2$ are 
a combination of a horizontal translation and possibly the reflection about the line $\R \times \{\frac{\pi}{2}\}$,
in which case we are forced to have $X_1=X_2$ as desired.

Even more basic is that $M_1$ and $M_2$ have to have the same topology. This further separates (Bii) and (Biii) from the others.

To distinguish (Bi) from (Ci) and (Cii), observe that in the latter cases $X$ has a fixed point, but in case (Bi) it does not. 
To distinguish between (Ci) and (Cii), we note that a conformal isomorphism must fix the origin, and so is a combination of a rotation and homothety, but any of the vector fields in each of these cases is invariant under the push-forward by such a transformation. As a by-product, we find also that $X_1=X_2$.

Finally, in cases (Bii) and (Biii), automorphisms of $\R \times S^1$ are generated by
translations and reflections in both $x$ and $\th$. 
Since we always have $\al > 0$, reflections reversing the orientation of the line are not permitted, and any automorphism must be a combination of a translation (in $x$ and $\th$) and possibly a reflection of the form 
$(x,\th) \mapsto (x,-\th)$.

In case (Bii), any vector field $\frac{1}{\al} \pl{}{x}$ is invariant under the push forward by such an automorphism. Finally, since $\be > 0$ in case (Biii), the only permitted automorphisms are translations. But any vector field $\frac{1}{\al} \pl{}{x} + \frac{\be}{\al} \pl{}{\th}$ is invariant under the push-forward by a translation.

We have thus shown that the two isomorphic measure expanders are in the same family, and have the same vector field $X$. It remains to show that the measures of the measure expanders are related in the way claimed in the theorem.

\begin{lem}\label{lem_split_isom}
Let $N$ be either $S^1$ or a connected open subset of $\R$,
with their standard metrics.
Suppose that $\nu_1,\nu_2\in \mathcal{R}(N)$, $\al > 0$, and
that the corresponding nontrivial Radon measures
$$\mu_i = e^{\al x} dx \otimes \nu_i, \quad i  \in \{1,2\},$$
on $\R \times N$ are expanding under the action of the translating vector field $\frac{1}{\al} \pl{}{x}$. If $\R \times N$ has the conformal structure of the Cartesian product, then the measure expanders $(\R \times N,\mu_i, \frac{1}{\al} \pl{}{x})$ for $i=1,2$ are isomorphic if and only if there exists an isometry $\varphi:N \to N$, and some constant $\lambda > 0$, such that \begin{equation}\label{eqn scaled isom}
    \varphi_*(\nu_1) = \lambda \cdot \nu_2.
\end{equation}
%
\end{lem}

\begin{proof}
If $\Phi: \R \times N \to \R \times N$ is a conformal diffeomorphism such that $\Phi_*(\frac{1}{\al} \pl{}{x})=\frac{1}{\al}\pl{}{x}$,
then $\Phi_*(\pl{}{y})=\pm \pl{}{y}$. Therefore, 
$$ \Phi(x,y) := (x + {\al}^{-1} \log(\lambda) , \varphi(y)), \quad \forall (x,y) \in \R \times N,$$
for some $\lambda>0$, and $\varphi: N \to N$ an isometry. If $\Phi$ is additionally an isomorphism between the measure expanders
$(\R\times N, \mu_i,\frac{1}{\al}\pl{}{x})$, then
\begin{multline*}
e^{\al x} dx\tensor \nu_2=\mu_2=\Phi_*(\mu_1)
= \Phi_*( e^{\al x} dx \otimes \nu_1)\\
=e^{\al x - \log \lambda} dx \otimes \varphi_*(\nu_1) = e^{\al x} dx \otimes \lambda^{-1} \varphi_*(\nu_1),
\end{multline*}
and so $\nu_1$ and $\nu_2$ must satisfy $\eqref{eqn scaled isom}$.
\end{proof}

The last step in the proof of Theorem~\ref{meas_exp_class} is to show that the measures are related in the way described in the theorem. For cases (A), (Bi), and (Bii), it follows immediately from Lemma~\ref{lem_split_isom}.

In case (Biii), recall that any conformal automorphism of $\R \times S^1$ fixing a vector field of the form $$\frac{1}{\al}\pl{}{x} + \frac{\be}{\al}\pl{}{\th}, \quad \al, \be > 0,$$
must be a translation. Since the conjugation of a translation by a twisting diffeomorphism $F_{\be}$
is still a translation, any two isomorphic measure expanders $(F_{\be}^*(\R \times S^1) , e^{\al x} dx \otimes \nu_i , \frac{1}{\al} \pl{}{x})$, must be isomorphic via a translation. Repeating the calculation in the proof of Lemma~\ref{lem_split_isom}, we deduce that there is some $\la > 0$ and an \emph{orientation preserving} isometry $\vph : S^1 \to S^1$ such that $\nu_1$ and $\nu_2$ satisfy \eqref{eqn scaled isom}.

For cases (Ci) and (Cii), any conformal automorphism fixing the vector field must also fix the origin, and hence will restrict to a conformal automorphism on the punctured plane. These cases therefore follow from (Bii) and (Biii).

Combining the entirety of Section \ref{class_meas_exp_sect}, this completes the proof of Theorem \ref{meas_exp_class}.

\section{Gradient expanding solitons}
\label{grad_sect}

Suppose $(M,g,\nabla f)$ is an expanding gradient Ricci soliton. If we rotate the vector field by 90 degrees then we get a 
Killing field; see e.g. \cite[Lemma 3.1]{chow_soliton_book}. Therefore, any measure expander that corresponds to a gradient soliton must have the property that this rotated vector field leaves the measure invariant. Since Killing fields on complete manifolds are necessarily complete themselves, this immediately rules out any of the measure expanders from (A) of Theorem \ref{meas_exp_class}.

In cases (Biii) and (Cii), for the corresponding soliton to be gradient, the rotated vector field $-\frac{\beta}{\al} \pl{}{x} + \frac{1}{\al} \pl{}{\th}$ must be a Killing field, which has flow $f_s(x,\th) := (x - \frac{\be s}{\al} , \th + \frac{s}{\al})$ for $s \in \R$. Since the measure $\mu$ must be invariant under this flow, we can deduce that
\begin{equation}\label{eqn twist killing}
\mu \left( (0,1) \times S^1\right) = (f_\al)_* \mu \left( (0,1) \times S^1\right) = \mu \left( (\be , 1 + \be) \times S^1 \right).
\end{equation}
Furthermore, using that $\mu$ has the form
$$ \mu  = (F_{\be})_* \left( e^{\al x} dx \otimes \nu \right),$$
for some $\nu \in \mathcal{R}(S^1)$, and that $F_{\beta}^{-1}\left( I \times S^1 \right) = I \times S^1$ for any interval $I$, when substituted into \eqref{eqn twist killing} we deduce that
$$ \frac{1}{\al} (e^{\al}-1) \nu(S^1) = \mu \left( (0,1) \times S^1\right) = \mu \left( (\be,1+\be) \times S^1\right) = \frac{1}{\al}  e^{\al \be}(e^{\al}-1) \nu(S^1), $$
and hence $\nu = 0$, which is a contradiction. Therefore, the only cases in Theorem~\ref{meas_exp_class} that could correspond to gradient solitons are (Bi), (Bii) and (Ci).

In case (Bi), the rotated vector field $\pl{}{y}$ must be a Killing field, and so the fibre measure $\nu$ must be invariant under translation. Therefore, up to isomorphism of the measure expander, $\nu$ is just the Lebesgue measure $dy$ on $\R$.

Similarly, in cases (Bii) and (Ci) the fibre measure $\nu$ must be invariant under rotation, and so up to isomorphism of the measure expander is just the quotient of the Lebesgue measure $d\theta$ on $S^1$.

We have proved the following.

\begin{thm}
\label{thm grad solitons}
Suppose that $(M,g,\nabla f)$ is a nontrivial expanding gradient Ricci soliton, and $(M,\mu,\nabla f)$ the corresponding measure expander as in Theorem~\ref{exp_meas_corresp}. Then $(M,\mu,\nabla f)$ is isomorphic to one of the following distinct measure expanders:

\begin{enumerate}
\item[(1)] $\left(\C, e^x dx \otimes dy, \frac{\partial}{\partial x}\right)$, inducing the universal cover of the soliton emanating from the punctured plane.

\item[(2)] $\left(\R \times S^1, e^{\al x} dx \otimes d \th, \frac1{\al}\pl{}{x}
\right)$, for each $\al > 0$, inducing solitons with one cusp end and one conical end.

\item[(3)] $\left(\C,e^{\al x} dx \otimes d\th, \frac1{\al}\pl{}{x}
\right)$, for each $\al > 0$,
inducing  solitons having one conical end, with $\al = 2$ giving the Gaussian soliton.
\end{enumerate}
\end{thm}

The argument above shows that any nontrivial expanding gradient soliton is induced by one of the examples (1), (2) or (3), but it is worth verifying that every measure expander in this list really does induce a \emph{gradient} soliton. This follows because the solitons will inherit the invariance under vertical translations from the measures $\mu$, and this allows us to define a soliton potential function $f$, depending on $x$ but not on $y$ (or $\th$), 
simply by integrating the ODE $\grad f=X$ along a horizontal line.

This recovers the known classification of expanding 
\emph{gradient} solitons in two dimensions (e.g. \cite[Chapter 3]{chow_soliton_book}).

\vskip 0.2cm

\noindent


LP: \href{mailto:lukepeachey@cuhk.edu.hk}{lukepeachey@cuhk.edu.hk}


\noindent
{\sc Department of Mathematics, Chinese University of Hong Kong,  Shatin, New Territories, HK.}

\noindent


PT: \href{mailto:p.m.topping@warwick.ac.uk}{p.m.topping@warwick.ac.uk}

\noindent
{\sc Mathematics Institute, University of Warwick, Coventry,
CV4 7AL, UK.}

\end{document}